\documentclass{amsart}
\usepackage[T1]{fontenc}
\usepackage[ansinew]{inputenc}
\usepackage{amsmath}
\usepackage{amsthm}
\usepackage{amssymb}
\usepackage{hyperref}
\usepackage[english]{babel}
\usepackage{setspace}

\newtheorem{thm}{Theorem}
\newtheorem{cor}[thm]{Corollary}

\theoremstyle{remark}

\theoremstyle{definition}
\newtheorem{defn}[thm]{Definition}

\newtheorem{rmk}[thm]{Remark}

\def\beq{\begin{equation}}
\def\eeq{\end{equation}}

\def\crash#1{}

\def\N{{\mathbb N}}
\def\Z{{\mathbb Z}}

\def\R{{\mathbb R}}
\def\C{{\mathbb C}}
\def\O{\Omega}

\def\l{\left}
\def\r{\right}
\def\[[{\l[\l[}
\def\]]{\r]\r]}
\def\p{\prime}

\def\sgq{\sigma_q}

\def\ord{{\rm ord}}

\def\cf{\emph{cf. }}
\def\ie{\emph{i.e. }}
\def\ds{\displaystyle}

\def\cL{{\mathcal L}}

\def\cF{{\mathcal F}}
\def\cH{{\mathcal H}}
\def\cJ{{\mathcal J}}

\def\la{\lambda}

\author{Lucia Di Vizio}
\thanks{Institut de Math\'{e}matiques de Jussieu,
Topologie et g\'{e}om\'{e}trie alg\'{e}briques,
Case 7012,
2, place Jussieu,
75251 Paris Cedex 05, France. e-mail: {\tt divizio@math.jussieu.fr}}

\date{~}

\title[An ultrametric version of the Maillet-Malgrange theorem...]
{An ultrametric version of the Maillet-Malgrange theorem for
nonlinear $q$-difference equations}

\begin{document}

\maketitle
\begin{abstract}
We prove an ultrametric $q$-difference version of the
Maillet-Malgrange theorem, on the Gevrey nature of formal solutions of
nonlinear analytic $q$-difference equations. Since $\deg_q$ and
$\ord_q$ define two valuations on $\C(q)$, we obtain, in particular,
a result on the growth of the degree in $q$ and the order at $q$ of
formal solutions of nonlinear $q$-difference equations, when $q$ is
a parameter. We illustrate the main theorem by considering two
examples: a $q$-deformation of ``Painlev\'{e} II'', for the nonlinear
situation, and a $q$-difference equation satisfied by the colored
Jones polynomials of the figure $8$ knots, in the linear case.
\par
We also consider a $q$-analog of the Maillet-Malgrange theorem, both in the complex
and in the ultrametric setting, under the assumption that $|q|=1$ and a classical
diophantine condition.
\end{abstract}


\section*{Introduction}
\addcontentsline{toc}{section}{Introduction}

In 1903, E. Maillet  \cite{Maillet} proved that a formal power series solution of an algebraic
differential equation is Gevrey. B. Malgrange \cite{malgrangemaillet}
generalized and made more precise Maillet's statement in
the case of an analytic nonlinear differential equation.
Finally C. Zhang \cite{changgui} proved a
$q$-difference-differential version of the Maillet-Malgrange theorem.
In the meantime a Gevrey theory for linear $q$-difference-differential equations has
been largely developed; \cf for instance \cite{Ramis}, \cite{Bezivin},
\cite{naegele}, \cite{Fleinert}.
\par
In this paper we prove an analogue of the Maillet-Malgrange theorem for
ultrametric nonlinear analytic $q$-difference equations, under the
assumption $|q|\neq 1$. It generalizes to nonlinear $q$-difference equations a
theorem of  B\'{e}zivin and Boutabaa; \cf
\cite{BezivinBoutabaa}. The proof follows \cite{malgrangemaillet}.
\par
The same technique allows to prove a Maillet-Malgrange theorem for
$q$-difference equations when $|q|=1$, both in the
complex and in the ultrametric setting, under a classical diophantine hypothesis:
this result generalizes the main result of \cite{Bezivin-nonlin} and answers a question
asked therein. Notice that the problem of nonlinear differential
equation in the ultrametric setting is treated in
\cite{SS1},\cite{SS2},\cite{SS3}, where a $p$-adic avatar of diophantine conditions
on the
exponents is also assumed.
\par
One of the reasons that makes the ultrametric statement interesting is the
possible application to the case when $q$ is a parameter (\cf \S\ref{sec:examples} below).
For instance, when $q$ is a parameter, Corollary \ref{cor:q-adicmailletmalgrangevero}
(\cf below) becomes:

\begin{thm}\label{thm:intro}
Suppose that we are given a nontrivial algebraic nonlinear $q$-difference equation
$$
F(q,x,y(x),\dots,y(q^nx))=0\,,
$$
\ie $F(q,x,w_0,\dots,w_n)\in\C[q,x,w_0,\dots,w_n]$ nonidentically zero, with a
formal solution $y(x)=\sum_{h\geq 0}y_hx^h\in\C(q)\[[x\]]$.
Then there exist nonnegative numbers $s,s^\p$ such that
$$
\limsup_{h\to\infty}\frac{1}{h}\l(\deg_q y_h-s\frac{h(h-1)}{2}\r)<+\infty
$$
and
$$
\limsup_{h\to\infty}\frac{1}{h}\l(\ord_q y_h-s^\p\frac{h(h-1)}{2}\r)>-\infty\,.
$$
\end{thm}

\medskip
We could give a more precise statement in which
$1/s$ and $-1/s^\p$ (with the convention $1/0=+\infty$) are slopes of the
Newton polygon of the linearized $q$-difference operator of
$F(q,x,y(x),\dots,y(q^nx))=0$ along $y(x)$ (\cf Theorem \ref{thm:q-adicmailletmalgrange2}).
\par
In classical literature on special functions, $q$ is frequently a
parameter. Basic hypergeometric equations are the most classical
example in the linear case, while the $q$-analogue of Painlev\'{e} equations
are nonlinear examples, that has been largely studied in the last
years.\footnote{For some examples of formal solutions of Painlev\'{e}
equations \cf for instance \cite{Ramani}.} This ultrametric
\emph{``$q$-adic''} approach to the study of a family of functional
equations depending on a parameter is peculiar to $q$-difference
equations.

\paragraph*{Acknowledgement.}
I would like to thank Changgui Zhang, whose questions are at the origin of this paper, and
Jean-Paul B\'{e}zivin for his attentive reading of the manuscript and his
numerous interesting comments.

\section{Ultrametric $q$-analog of the Maillet-Malgrange theorem for $|q|\neq 1$}
\label{sec:statement}

Let $\Omega$ be a complete ultrametric valued field, equipped with the ultrametric norm $|~|$, and let
\emph{$q\in\Omega$ be an element of norm strictly greater that $1$}.\footnote{We could have chosen the
opposite convention $|q|<1$, which leads to analogous statements.}

\subsection{Digression on the linear case}

We denote by $\O\{x\}$ the ring of germs of analytic functions at $0$ with coefficients in $\O$,
\ie the convergent elements of $\O\[[x\]]$.
To the linear $q$-difference equation
$$
\cL y(x)=\sum_{i=0}^n a_i(x)y(q^i x)=0\,,
$$
with $a_i(x)\in\O\{x\}$, we can attach the Newton polygon
\beq\label{eq:NP}
N_q(\cL)= \hbox{ convex envelop}\l(\mathop\cup_{i=0}^n
\l\{(i,h):h\geq\ord_{x=0}a_i(x)\r\}\r)\,.
\eeq
Of course the polygon $N_q(\cL)$ has a finite number of finite
sides, with \emph{rational} slopes, plus two infinite vertical sides.
We adopt the convention that the right vertical side has slope $+\infty$ and
the left one has slope $-\infty$.
\par
B\'{e}zivin and Boutabaa have proved the following result:

\begin{thm}{\rm (\cite{BezivinBoutabaa})}
Let $g(x)\in\O\{x\}$ and $y(x)=\sum_{h\geq 0}y_h x^h\in\O\[[x\]]$ be such that $\cL y(x)=g(x)$.
Then either $y(x)\in\O\{x\}$ or there exists a positive slope $r\in]0,+\infty[$ of $N_q(\cL)$
such that
$$
\sum_{h\geq 0}\frac{y_h}{q^{\frac{h(h-1)}{2r}}}x^h\,,
$$
is a convergent \emph{nonentire} series.
\end{thm}

\subsection{Statement of the main result}

Consider an analytic function at $0$ of $n+2$ variable, \ie a power series
$$
F(x,w_0,w_1,\dots,w_n)=\sum_{k,k_0,\dots,k_n>0}A_{k,k_0,\dots,k_n}x^kw_0^{k_0}\cdots w_n^{k_n}
\in\Omega\[[x,w_0,w_1,\dots,w_n\]]\,,
$$
such that
$$
\limsup_{k+\sum_{i=0}^nk_i\to \infty}\l|A_{k,k_0,\dots,k_n}\r|^{\frac{1}{k+\sum_{i=0}^nk_i}}<+\infty\,.
$$
Remark that we have assumed, with no loss of generality,
that $F(0,\dots,0)=0$.
We are interested in studying formal solutions of the
\emph{nonlinear analytic $q$-difference equation}
\beq\label{eq:nonlineq}
F(x,\varphi(x),\varphi(qx),\dots,\varphi(q^nx))=0\,. \eeq To
simplify notation for any $\varphi\in\Omega\[[x\]]$ we set
$\Phi=(\varphi(x),\varphi(qx),\dots,\varphi(q^nx))$, and we denote
by $\sgq$ the usual $q$-difference operator acting on $\O\[[x\]]$:
$$
\begin{array}{rccc}
\sgq:&\O\[[x\]]&\longrightarrow&\O\[[x\]]\,,\\
&\varphi(x)&\longmapsto&\varphi(qx)\,.
\end{array}
$$
For any formal power series $\varphi(x)\in\Omega\[[x\]]$, such that
$\varphi(0)=0$, let $\cF_\varphi$ be the \emph{linearized
$q$-difference operator of $F$ along $\varphi$}:
$$
\cF_\varphi=\sum_{i=0}^n\frac{\partial F}{\partial w_i}(x,\Phi)\sgq^i\,.
$$
The operator $\cF_\varphi$ being linear, we can define its \emph{Newton polygon $N_q(\cF_\varphi)$}
in the
usual way (\cf equation \ref{eq:NP}).
We want to prove that, for a solution $\varphi(x)$ of (\ref{eq:nonlineq}),
the positive slopes of $N_q(\cF_\varphi)$ are linked to the
$q$-Gevrey order of $\varphi(x)$:

\begin{defn}
A formal power series $\varphi(x)=\sum_{h\geq 0}\varphi_h x^h\in\Omega\[[x\]]$
is a \emph{$q$-Gevrey series (of order $s\in\R$)} if the series
$$
\sum_{h\geq 0}\frac{\varphi_h}{q^{s\frac{h(h-1)}{2}}}x^h
$$
is convergent.
\end{defn}

We can state our main result:

\begin{thm}\label{thm:q-adicmailletmalgrange}
Let $\varphi(x)\in x\Omega\[[x\]]$ be a  formal solution of the equation (\ref{eq:nonlineq})
and let $r\in]0,+\infty]$ be the smallest positive slope of the Newton polygon of $\cF_{\varphi}$.
If $\frac{\partial F}{\partial w_n}(x,\Phi)\neq 0$, then $\varphi(x)$ is a $q$-Gevrey series
of order $1/r$.\footnote{We have implicitly set $1/+\infty=0$.}
\end{thm}

As a consequence we obtain:

\begin{cor}\label{cor:q-adicmailletmalgrangevero}
Let $\varphi(x)\in x\Omega\[[x\]]$ be a  formal solution of equation (\ref{eq:nonlineq}).
If $F(x,w_0,w_1,\dots,w_n)$ is not identically zero, then $\varphi(x)$ is a $q$-Gevrey series
(of some nonspecified order).
\end{cor}

\subsection{When $q$ is a parameter...}

Suppose that $F(q,x,w_0,\dots,w_n)\in\C[q,x,w_0,\dots,w_n]$, where $q$ is a parameter,
and that we have a formal solution
$\varphi(x)=\sum_{h\geq 0}y_hx^h\in\C(q)\[[x\]]$.\footnote{The results that follows are
actually true when we replace $\C$ by any field.} Up to equivalence, there are
exactly two ultrametric norm over $\C(q)$ such that $q$
has norm different than $1$. For any $f(q)\in\C[q]$
they are defined by
\begin{enumerate}
\item
$|f(q)|_{q^{-1}}=d^{-\deg_q f(x)}$;
\item
$|f(q)|_q=d^{\ord_q f(q)}$;
\end{enumerate}
where $d\in]0,1[$ is a fixed real number.
Of course, $|~|_q$ and $|~|_{q^{-1}}$ extends to $\C(q)$ by multiplicativity.
Notice that $|q|_q=d<1$ and $|q|_{q^{-1}}=d^{-1}>1$.
\par
Taking $\Omega$ to be the completion of $\C(q)$
with respect to $|~|_q$ (resp. $|~|_{q^{-1}}$), we immediately see that Theorem
\ref{thm:intro} is a particular case of Corollary \ref{cor:q-adicmailletmalgrangevero} and that
Theorem \ref{thm:q-adicmailletmalgrange} becomes:

\begin{thm}\label{thm:q-adicmailletmalgrange2}
Let
$$
\frac{\partial F}{\partial w_n}(q,x,y(x),\dots,y(q^nx))\neq 0\,.
$$
If $r\in]0,+\infty]$ (resp. $r^\p\in[-\infty, 0[$) is the smallest positive slope
(resp. the largest negative slope) of $\cF_\varphi$,
then
$$
\limsup_{h\to\infty}\frac{1}{h}\l(\deg_q y_h-s\frac{h(h-1)}{2}\r)<+\infty\,,
\hbox{ with $s=1/r$,}
$$
and
$$
\limsup_{h\to\infty}\frac{1}{h}\l(\ord_q y_h-s^\p\frac{h(h-1)}{2}\r)>-\infty\,,
\hbox{ with $s^\p=-1/r^\p$.}
$$
\end{thm}

\section{Examples}
\label{sec:examples}

\subsection{Colored Jones polynomial of figure $8$ knot}

We consider the $q$-difference equation satisfied by the generating function of
the sequence of invariants of the figure $8$ knot called the colored Jones polynomials
(\cf \cite[\S 3]{Garoufdeformationvar}):
$$
J(q,n)=\sum_{k=0}^n q^{nk}(q^{-n-1}; q^{-1})_k(q^{-n+1};q)_k\in\Z[q,q^{-1}]\,,
\ \forall n\in\N\,.
$$
The series $\cJ(x)=\sum_{n\geq 0}J(q,n)x^n\in\C(q)\[[x\]]$ satisfies the linear
$q$-difference equation
$$
\begin{array}{l}
\l[q\sgq(q^2+\sgq)(q^5-\sgq^2)(1-\sgq^2)\r]y(x)-\\ \\
x\Big[\sgq^{-1}(1+\sgq)\Big(q^4+\sgq\l(q^3-2q^4\r)+\sgq^2\l(-q^3+q^4-q^5\r)\\ \\
\hskip 30 pt+\sgq^3\l(-2q^4+q^5\r)+\sgq^4q^4\Big)(q^5-q^2\sgq^2)(1-\sgq)\Big]y(x)+\\ \\
x^2\Big[q^5(1-\sgq)(1+\sgq)(1-q^3\sgq^2)\l(q^8+\sgq(q^9-2q^8)-\sgq^2(-q^7+q^8-q^9)+q^7\sgq^3+q^8\sgq^4\r)\Big]y(x)-\\ \\
x^3\l[q^{10}\sgq(1-\sgq)(1+q^2\sgq)(1-q^5\sgq^2)\r]y(x)=0\,.
\end{array}
$$
The finite slopes of the Newton polygon
are: $-1/2$, $0$, $1/2$.
It is clear looking at the leading term of $J(q,n)$
that $\cJ(x)$ cannot be a converging series for the norms $|~|_q$ and
$|~|_{q^{-1}}$. Therefore it follows from B\'{e}zivin and Boutabaa theorem that
$$
\limsup_{n\to 0}\frac{1}{n}\l(\deg_q J(q,n)-2\frac{n(n-1)}{2}\r)<+\infty
$$
and
$$
\limsup_{n\to 0}\frac{1}{n}\l(\ord_q J(q,n)+2\frac{n(n-1)}{2}\r)>-\infty\,.
$$
Notice that modulo the AJ conjecture (\cf \cite[\S 1.4]{Garoufdeformationvar}),
those slopes are the same as the ones defined in \cite{pentes}.

\subsection{A $q$-deformation of the second Painlev\'{e} equation}

Let us consider the nonlinear $q$-difference equation
associated to the analytic funtion at
$(0,1,1,1)$:\footnote{This example is studied in \cite[\S3.5]{KMNOY1} and
\cite[Eq.(2.55)]{Ramani}, where
many other examples can be found.}
$$
F(x,w_{-1},w_0,w_1)=(w_0+x)(w_0w_1-1)(w_0w_{-1}-1)-qx^2w_0\,,
$$
namely
\beq\label{eq:p2}
(y(x)+x)(y(x)y(qx)-1)(y(x)y(q^{-1}x)-1)-qx^2y(x)=0\,.
\eeq
It is a $q$-deformation of $P_{II}$. Let $\varphi(x)\in\C(q)\[[x\]]$, with $\varphi(0)=1$,
be a formal solution of equation (\ref{eq:p2}). Then
$$
\begin{array}{rcl}
\ds\cF_\varphi
&=&
\sum_{i=-1}^1\frac{\partial F}{\partial w_i}(x,\varphi(q^{-1}x),\varphi(x),\varphi(qx))
    \sgq^i\\ \\
&=&\Big[\big(\varphi(x)+x\big)\big(\varphi(x)\varphi(qx)-1\big)\varphi(x)\Big]\sgq^{-1}\\ \\
&+&\Big[\big(\varphi(x)\varphi(qx)-1\big)\big(\varphi(x)\varphi(q^{-1}x)-1\big)+
    \big(\varphi(x)+x\big)\varphi(qx)\big(\varphi(x)\varphi(q^{-1}x)-1\big)\\ \\
&&+\big(\varphi(x)+x\big)\big(\varphi(x)\varphi(qx)-1\big)\varphi(q^{-1}x)-qx^2\Big]\sgq^0\\ \\
&+&\Big[(\varphi(x)+x)\varphi(x)(\varphi(x)\varphi(q^{-1}x)-1)\Big]\sgq
\end{array}
$$
A formal solution of equation (\ref{eq:p2}) is give by
$$
\varphi(x)=\frac{{}_{1}\Phi_{1}(0;-q;q,-q^2x)}{{}_{1}\Phi_{1}(0;-q;q,-qx)}
=1+\frac{q}{1+q}x+\cdots\,,
$$
where ${}_{1}\Phi_{1}(0;-q;q,x)$ is a basic hypergeometric series:
$$
{}_{1}\Phi_{1}(0;-q;q,-qx)
=\sum_{h\geq 0}\frac{q^{h(h-1)}}{(-q;q)_h(q;q)_h}x^h\,,
$$
and
$$
(a;q)_h=(1-a)(1-aq)\dots(1-aq^{h-1})\,.
$$
A direct and straightforward calculation
shows that the Newton polygon of $\cF_\varphi$ is \emph{regular singular},
meaning that it has only one finite horizontal slope of length $2$,
plus the two vertical sides.
Therefore Theorem \ref{thm:q-adicmailletmalgrange}
implies that the solution $\varphi(x)=1+\sum_{h\geq 1}\varphi_hx^h$ considered
above verifies:
$$
\limsup_{h\to\infty}\frac{1}{h}\deg_q \varphi_h<+\infty
\hbox{ and }
\limsup_{h\to\infty}\frac{1}{h}\ord_q \varphi_h>-\infty\,.
$$
In other words, the solution $\varphi(x)\in\C(q)\[[x\]]$ is convergent for both the norm
$|~|_q$ and the norm $|~|_{q^{-1}}$.
\par
We could have also remarked that ${}_{1}\Phi_{1}(0;-q;q,x)$ is a solution of the linear equation
$$
\sgq^{-2}(\sgq-1)(\sgq+1) y(x)+q^2x y(x)=0\,,
$$
whose Newton polygon has only a horizontal finite slope. This means that
${}_{1}\Phi_{1}(0;-q;q,x)$ is convergent for both $|~|_q$ and $|~|_{q^{-1}}$, and hence
that $\varphi(x)$ is also convergent.

\section{Proofs}

\subsection{Proof of Theorem \ref{thm:q-adicmailletmalgrange}}

The proof follows \cite{malgrangemaillet}.
It relies on the ultrametric implicit function theorem;
\cf \cite{acampo}, \cite{serre}, \cite{SS2}.

\medskip
We set $\varphi(x)=\sum_{h\geq 1}\varphi_hx^h$. For any $k\in\N$, let
\begin{trivlist}
\item
1. $\varphi_k(x)=\sum_{h=0}^k\varphi_h x^h$;
\item
2. $\psi(x)$ be a formal power series such that $\varphi(x)=\varphi_k(x)+x^k\psi(x)$;
\item
3. $\Psi(x)=(\psi(x),\psi(qx),\dots,\psi(q^n x))$ and
$\Phi_k(x)=(\varphi_k(x),\varphi_k(qx),\dots,\varphi_k(q^n x))$.
\end{trivlist}
Let $W=(w_0,\dots,w_n)$, $Z=(z_0,\dots,z_n)$. By taking the Taylor expansion of
$F(x,W+Z)$ at $W$ we obtain:
$$
F(x,W+Z)=F(x,W)+\sum_{i=0}^n\frac{\partial F}{\partial w_i}(x,W)z_i
+\sum_{i,j=0}^nH_{i,j}(x,W,Z)z_iz_j\,,
$$
where $H(x,W,Z)$ is an
analytic function of $2n+3$ variables in a neighborhood of zero.
Hence we
can write:
\beq\label{eq:taylor}
\begin{array}{rcl}
0=F(x,\Phi)
&=&\ds F(x,\Phi_k(x))+x^k\sum_{i=0}^n\frac{\partial F}{\partial w_i}(x,\Phi_k)q^{ik}\sgq^i\psi\\
&&\ds +x^{2k}\sum_{i,j=0}^nH_{i,j}(x,\Phi_k(x),x^k\Psi(x))q^{(i+j)k}\sgq^i\psi\sgq^j\psi\,.
\end{array}
\eeq

To finish the proof we have to distinguish two cases: $r<+\infty$ and $r=+\infty$.

\subsubsection*{Case 1. $r<+\infty$}

We are going to choose $k\geq\sup(k_1,k_2+l+1)$, where $k_1,k_2,l$ are constructed as follows
(\cf figure below).
First of all let $(n^\p,l)\in\N^2$ be the point of $N_q(\cF_\varphi)$ which verifies the
two properties:
\begin{trivlist}
\item 1.
$l$ is the smallest real number such that $(j,l)\in N_q(\cF_\varphi)$ for some $j\in\R$;
\item 2.
$n^\p$ is the greatest real number such that $(n^\p,l)\in N_q(\cF_\varphi)$.
\end{trivlist}
Let us consider the polynomial
$$
\cL(T)=\sum_{i=0}^{n^\p}\l[\frac{1}{x^l}\frac{\partial F}{\partial w_i}(x,\Phi)\r]_{x=0}T^i\,.
$$
We chose $k_1$ to be a positive integer
such that for any $k\geq k_1$, the polynomial $\cL(T)$ does not vanishes at $q^k$,
and $k_2\geq r(n-n^\p)$. Notice that for any $k\geq k_2+l$, the smallest positive slope
of $N_q(\cF_{\varphi_k})$ is equal to $r$ and the point $(n^\p,l)$ is the ``lowest'' point of
$N_q(\cF_{\varphi_k})$ with greater abscissae.

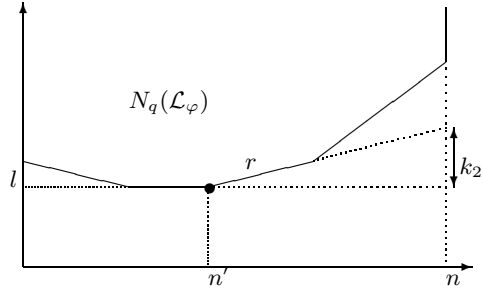
\begin{center}
\smallskip
\begin{picture}(160,100)
\put(0,0){\vector(0,1){100}}%
\put(0,0){\vector(1,0){170}}%
\put(0,40){\line(4,-1){40}}%
\put(40,30){\line(1,0){30}}%
\put(70,30){\line(4,1){40}}%
\put(110,40){\line(4,3){50}}%
\put(160,78){\line(0,1){20}}
\qbezier[90](0,30)(50,30)(160,30)
\put(-5,30){\hbox{\small $l$}}
\qbezier[20](70,0)(70,10)(70,30)
\put(70,-7){\hbox{\small $n^\p$}}
\qbezier[30](160,0)(160,31)(160,78)
\put(160,-7){\hbox{\small $n$}}
\qbezier[25](110,40)(130,45.5)(160,52.5)
\put(163,30){\vector(0,1){22.5}}
\put(163,30){\vector(0,-1){0}}
\put(165,35){\hbox{\small $k_2$}}
\put(84,37){\hbox{\small $r$}}
\put(40,60){\hbox{\small $N_q(\cL_\varphi)$}}
\put(68,27){$\bullet$}
\end{picture}
\bigskip
\end{center}

Remark that for any $k\geq\sup(k_1,k_2+l+1)$ we have
$$
\ord_{x=0}\sum_{i=0}^n\frac{\partial F}{\partial w_i}(x,\Phi_k)q^{ik}\sgq^i\psi
\geq\ord_{x=0}\psi(x)+\inf_{i=0,\dots,n}\ord_{x=0}\frac{\partial F}{\partial w_i}(x,\Phi_k)
\geq l+1\,.
$$
Therefore we can write the linear part of equation (\ref{eq:taylor}) in the form
$$
\frac{1}{x^l}
\sum_{i=0}^n\frac{\partial F}{\partial w_i}(x,\Phi_k)q^{ik}\sgq^i\psi
=\cL(q^k\sgq)\psi+x\widetilde\cL(x,\sgq)\psi\,,
$$
where $\widetilde\cL(x,\sgq)$ is an analytic functional.
Moreover we deduce from equation (\ref{eq:taylor}) that
$$
\ord_{x=0}F(x,\Phi_k)\geq k+l+1\,,
$$
so that there exists an analytic function $M(x,w_0,\dots,w_n)$ such that
equation (\ref{eq:taylor}) divided by $x^{l+k}$ becomes
\beq\label{eq:auxiliarfunctional}
\cL(q^k\sgq)\psi+x\widetilde\cL(x,\sgq)\psi+xM(x,x^k\Psi)=0\,.
\eeq
Since $\cL(q^k\sgq)$ is a linear operator with constant coefficients
and $\cL(q^h)\neq 0$ for any $h\geq k$,
equation (\ref{eq:auxiliarfunctional}) admits one unique formal solution
$\psi(x)\in x\Omega\[[x\]]$, whose coefficients can be constructed recursively.
\par
In order to conclude, we have to estimate the Gevrey order of $\psi(x)$.
Let us consider the following Banach $\Omega$-vector space:
$$
\cH_{s,m}=\l\{\sum_{h\geq 1}\varphi_hx^h\in\Omega\[[x\]]:
\sup_{h\geq 1}|\varphi_h||q|^{hm-s\frac{h(h-1)}{2}}<+\infty\r\}
$$
equipped with the norm
$$
\l\|\sum_{h\geq 1}\varphi_hx^h\r\|_{s,m}=\sup_{h\geq 1}|\varphi_h||q|^{hm-s\frac{h(h-1)}{2}}\,.
$$
Since for any positive rational number $s$ and any pair of positive integers $k,h$
we have
$$
|q|^{s\frac{k(k-1)}{2}}|q|^{s\frac{k(k-1)}{2}}
\leq |q|^{s\frac{(k+h)(k+h-1)}{2}}\,,
$$
the analytic functional
$$
A(\la,\psi)=\cL(q^k\sgq)\psi+\la x\widetilde\cL(\la x,\sgq)\psi+\la xM(\la x,\la^kx^k\Psi)\,
$$
is defined over $\O\times\cH_{s,n^\p}$:
$$
A(\la,\psi):\O\times\cH_{s,n^\p}\longrightarrow\cH_{s,0}\,,
$$
and verifies
$$
\hbox{$A(0,0)=0$ and }\frac{\partial A}{\partial \psi}(0,0)=\cL(q^k\sgq)\,.
$$
Since $\cL(q^k\sgq)$ is invertible, the implicit function theorem implies that
for any $\la$ in a neighborhood
of $0$ there exists $\psi_\la$ such that $A(\la,\psi_\la)=0$.
The formal solution $\psi$ of equation (\ref{eq:auxiliarfunctional}) being unique,
we must have $\psi_\la(x)=\psi(\la x)$ for any $\la$ closed to $0$, which ends the proof.

\subsubsection*{Case 2. $r=+\infty$}

We chose the point $(n^\p,l)$ as in the previous case: since there are no finite positive
slopes, we have $n^\p=n$. We can define the polynomial $\cL(T)$ in the same way as before.
So we choose $k_1\in\N$ such that $\cL(q^k)\neq 0$ for any $k\geq k_1$ and $k_2\in\N$ such that
$$
\inf_{i=0,\dots,n}\ord_{x=0}\l(\frac{\partial F}{\partial w_i}(x,\Phi_k)\r)>l
$$
for any $k\geq k_2$.
We deduce that $\ord_{x=0}F(x,\Phi_k)\geq k+l+1$ and hence we are reduced, by dividing equation
(\ref{eq:taylor}) by $x^{l+k}$, to consider the functional
$$
\cL(q^k\sgq)+\la xM(\la x,\la^kx^k\Psi)=0\,.
$$
The same argument as above also allows us to conclude the proof in this case.

\subsection{Proof of Corollary \ref{cor:q-adicmailletmalgrangevero}}

Following \cite{malgrangemaillet},
we are going to show by induction on $n$ that Theorem \ref{thm:q-adicmailletmalgrange} implies Corollary
\ref{cor:q-adicmailletmalgrangevero}. Notice that for $n=0$ we are in the classical case of Puiseux
development of a solution of an algebraic equation (\cf \cite{malgrangemaillet}).
So let us suppose $n\geq 1$.
\par
If there exists a positive integer $k$ such that
\beq\label{eq:hyp}
\frac{\partial^k F}{\partial w_n^k}(x,\Phi)\neq 0\,,
\eeq
we conclude by applying
Theorem \ref{thm:q-adicmailletmalgrange} to the $q$-difference equation
$$
\frac{\partial^{\kappa-1} F}{\partial w_n^{\kappa-1}}(x,\Phi)=0,
$$
where $\kappa$ is the smallest positive integer verifying equation (\ref{eq:hyp}).
\par
We now suppose that for any positive integer $k$ we have
$\frac{\partial^k F}{\partial w_n^k}(x,\Phi)=0$. By taking the Taylor expansion of
$F(x,w_0,\dots,w_n)$, we can verify that
$F(x,\varphi(x),\dots,\varphi(q^{n-1}x),\psi(x))\equiv0$
for any $\psi(x)\in x\O\[[x\]]$. In particular, there exists $\la\in\O$ such that
$F(x,w_0,\dots,w_{n-1},\la x)$ is not identically zero and
$F(x,\varphi(x),\dots,\varphi(q^{n-1}x),\la x)=0$.
So we are reduced to the case ``$n-1$''.

\section{Complex $q$-analog of the Maillet-Malgrange theorem for $|q|=1$}

Let $\O$ be either the ultrametric field defined in \S\ref{sec:statement}
or the complex field $\C$. We choose $q\in\O$ such that $|q|=1$ and $q$ is
not a root of unity.
\par
To the linear $q$-difference equation
$$
\cL y(x)=\sum_{i=0}^n a_i(x)y(q^i x)=0\,,
$$
with $a_i(x)=a_{i,j_i}x^{j_i}+a_{i,j_i+1}x^{j_i+1}+\dots\in\O\{x\}$,
we can attach a polynomial
$$
Q_\cL(T)=(T-1)\sum_{i=0}^n a_{i,j_i}T^i\,.
$$
We recall the result:

\begin{thm}(\cf \cite[Thm. 6.1]{Bezivin} and \cite[Thm. 6.1]{BezivinBoutabaa})
Let $\varphi(x)\in\O\[[x\]]$ be a formal solution of $\cL y(x)=0$.
We suppose that
\begin{quote}
\begin{itemize}
\item[$({\mathcal H})$]
There exist two constants $c_1,c_2>0$,
such that for any root $u$ of $Q_\cL(T)$ and any $n>>0$ the following inequality is satisfied:
$|q^n-u|\geq c_1 n^{-c_1}$.
\end{itemize}
\end{quote}
Then $\varphi(x)$ is convergent.
\end{thm}

In the nonlinear case we have the following result that generalizes
\cite[\S1]{Bezivin-nonlin}:

\begin{thm}\label{thm:q=1adicmailletmalgrange}
Let $\varphi(x)\in x\Omega\[[x\]]$ be a  formal solution of
the $q$-difference equation
\beq\label{eq:nonlineq=1}
F(x,\varphi(x),\varphi(qx),\dots,\varphi(q^nx))=0\,,
\eeq
analytic at zero.
We make the following assumptions:
\begin{enumerate}

\item
$\frac{\partial F}{\partial w_n}(x,\Phi)\neq 0$,

\item
the polynomial $Q_{\cF_\varphi}$ associated to
the linear operator $\cF_\varphi$ verifies the hypothesis $({\mathcal H})$.
\end{enumerate}
Then $\varphi(x)$ is convergent.
\end{thm}

\begin{rmk}
Notice that the second hypothesis is always verified in the following cases:
\begin{itemize}

\item
if $\Omega=\C$ and $q$ and the coefficients of $Q$ are algebraic numbers
(\cf \cite[2.2]{Bezivin-nonlin}),

\item
if $\O$ is an extension of a number field $K$ equipped with a $p$-adic valuation,
and $q$ and the roots of $Q(T)$ are in $K$
(in this case it is a consequence of Baker's theorem; \cf for instance
\cite[\S8.3]{DV-Inv-math})

\end{itemize}
\end{rmk}

\begin{proof}[Proof of Theorem \ref{thm:q=1adicmailletmalgrange}.]
The first part of the proof of Theorem \ref{thm:q-adicmailletmalgrange} is completely formal.
So once again we are reduced to consider equation (\ref{eq:auxiliarfunctional})
$$
\cL(q^k\sgq)\psi+x\widetilde\cL(x,\sgq)\psi+xM(x,x^k\Psi)=0\,.
$$
The key-point is the choice of $k>>1$, so that the
Newton polygon of the $q$-difference operator
$\cL(q^k\sgq)+x\widetilde\cL(x,\sgq)$ coincides with the Newton polygon of
$\cF_\varphi$, up to a vertical shift.

\par
Let $\cH(0,r)$ be the Banach algebra of analytic functions
converging over the closed disk $D(0,r^+)$ of center $0$ and radius
$r>0$, for $r$ small enough, equipped with the norm
$$
\l|\sum_{n\geq 0}a_n X^n\r|_{\cH(0,r)}=\sup_{n\geq 0}|a_n|r^n\,.
$$

It follows from \cite[Thm. 6.1]{Bezivin} and
\cite[Thm. 6.1]{BezivinBoutabaa}\footnote{Notice that
\cite[Thm. 6.1]{BezivinBoutabaa} is formulated only
for $q$-difference equations with polynomial coefficients,
but the same proof as \cite[Thm. 6.1]{Bezivin} works in the analytic case.}
that the operator
$\cL(q^k\sgq)+x\widetilde\cL(\la x,\sgq)$ acts on $\O\times \cH(0,r)$
and hence
$$
A(\la,\psi):\O\times\cH(0,r)\longrightarrow \cH(0,r)\,.
$$
The implicit function theorem also allows us to conclude this case.
\end{proof}

\newcommand{\etalchar}[1]{$^{#1}$}

\end{document}